\documentclass[11pt]{article}

\usepackage{latexsym}
\usepackage{amssymb}
\usepackage{amsfonts}
\usepackage{amsfonts}
\usepackage{amsmath}
\usepackage{amsthm}
\usepackage{amscd}
\usepackage{graphicx,float}
\usepackage{tikz}

\newtheorem{thm}{Theorem}

\newtheorem{lem}[thm]{Lemma}
\newtheorem{prop}[thm]{Proposition}
\theoremstyle{definition}
\newtheorem{defn}[thm]{Definition}
\newtheorem{rem}[thm]{Remark}

\newcommand{\Rd}{\mathbb{R}^{d}}

\newcommand{\Rdd}{\mathbb{R}^{2d}}
\def\R2d{{\mathbb R^{2d}}}
\newcommand\Wig{\mathop{\rm Wig}}
\def\supp{{\mathop{\rm supp\,}}}
\newcommand{\afrac}[2]{\genfrac{}{}{0pt}{1}{#1}{#2}}

\begin{document}
\title{\Large\bf Cohen class of time-frequency representations and operators: boundedness and uncertainty principles}
\author{{\large Paolo Boggiatto\footnote{e-mail: paolo.boggiatto@unito.it} , Evanthia Carypis\footnote{e-mail:
evanthia.carypis@unito.it} \ and Alessandro Oliaro\footnote{e-mail: alessandro.oliaro@unito.it}} \\
\it Department of Mathematics, University of Torino,
Italy}
\date{}

\maketitle

\begin{abstract}
This paper presents a proof of an uncertainty principle of
Donoho-Stark type involving $\varepsilon$-concentration of
localization operators. More general operators associated with
time-frequency representations in the Cohen class are then
considered. For these operators, which include all usual quantizations, we prove a boundedness result in the $L^p$
functional setting and a form of uncertainty principle analogous to that for localization operators.

\end{abstract}

\section{Introduction}
Uncertainty principles (UP) appear in harmonic analysis and signal
theory in a variety of different forms involving not only the couple
$(f,\widehat f)$ formed by a signal (function or distribution) and
its Fourier transform, but essentially every representation of a
signal in the time-frequency space. Among the wide literature on this topic we refer for example to \cite{Ben85, BogCarOli2012, BogCarOli2014, BogFerGal, Coh95, DonSta, FerGal10, FolSit97, Gro2003, Jan85, Jan98}.

In this paper we consider the case where the couple $(f,\widehat f)$
is substituted by a couple $(T_1f,T_2f)$, where $T_1, T_2$ are
operators by which, in some sense, the concentration of the signal
$f$ is ``tested''. The consequent uncertainty statement is then of
the following type: if the tests yield functions which are
sufficiently concentrated on some domains of the time-frequency
space, then the Lebesgue measure of these domains can not be ``too
small''.

We make now precise the type of operators that are used, in which
sense ``concentration'' is intended, and what is meant by ``too
small''.

The class of operators that we consider is strictly connected with
the \emph{Cohen class} of time-frequency representations, which
consists of sesquilinear forms of the type
\begin{equation}\label{Cohen}
Q_{\sigma}(f,g)(x,\omega) = \sigma\ast \Wig(f,g)(x,\omega),
\end{equation}
where
\begin{equation}\label{wig} \Wig(f,g)(x,\omega) =
\int_{\Rd}{e^{-2\pi i \omega\cdot
t}f\left(x+t/2\right)\overline{g\left(x-t/2\right)}\, dt}
\end{equation}
is the \emph{Wigner transform} and $\sigma$ is the \emph{Cohen
kernel}.  We shall shortly write $Q_\sigma (f)$ for the quadratic form $Q_\sigma(f,f)$.
Clearly the signals $f,g$ must be chosen in functional or
distributional spaces such that the convolution \eqref{Cohen} makes
sense.

The Cohen class finds its justification in applied signal
analysis as it actually coincides with
the class of quadratic {\it covariant} time-frequency representations.
More precisely, let $Q$ be any sesquilinear form (non a priori in
the Cohen class); a very natural
requirement is that a translation in time $\tau_af(x)=f(x-a)$ of the
signal should reflect into the same translation of its representation
along the time-axis, i.e. $Q(\tau_a f)(x,\omega)=Qf(x-a,\omega)$. On
the other hand a modulation $\mu_b f(x)=e^{2\pi ibx}f(x)$ should
reflect into a translation by the same parameter $b$ along the
frequency-axis, i.e. $Q(\mu_b f)(x,\omega)=Qf(x,\omega-b)$. It can be
proved that these two requirements, called {\it
covariance property}, actually characterize, under some minor technical
hypothesis, the Cohen class among all quadratic
representations (see e.g. \cite{Gro01-1}, Thm 4.5.1).

As described in \cite{BogDedOli2009}, \cite{BogDedOli2010}, we can
associate an operator $T^a_\sigma$, depending on a symbol $a$, with
each time-frequency representation $Q_\sigma$, by the formula:
\begin{equation}\label{QT}
(T_{\sigma}^{a}f,g) = (a, Q_{\sigma}(g,f)).
\end{equation}
Formula \eqref{QT} can be understood, e.g. in the Lebesgue setting, as follows:
$$
\begin{array}{l}
Q_{\sigma}: L^q(\Rd)\times L^p(\Rd) \to L^r(\Rdd), \\
T_\sigma:a\in L^{r'}(\Rdd)\to B\big(L^p(\Rd),L^{q'}(\Rd)\big),
\end{array}
$$
where $1< q <\infty,\ 1< r\leq \infty$, $1\leq p\leq\infty$, and
$\frac{1}{q}+\frac{1}{q'} = \frac{1}{r} +\frac{1}{r'} = 1$.
For simplicity we write $\Wig{}{}(f)$ and $Q_\sigma(f)$ when $f=g$.

\noindent More generally, if $\sigma\in \mathcal S'(\Rdd)$, formula
(\ref{QT}) defines a continuous linear map $T_\sigma^a: \mathcal
S(\Rd)\to \mathcal S'(\Rd)$,  and actually it establishes a
bijection between operators and sesquilinear forms, we refer to
\cite{BogDedOli2009} for details and general functional settings.
The operators $T_\sigma^a$, obtained by \eqref{QT} in correspondence
with representations $Q_\sigma$ in the Cohen class, will be called
{\it Cohen operators}. Referring to \eqref{QT}, we
actually remark that $(T_{\sigma}^{a}f,g) = (a,
Q_{\sigma}(g,f))=((a*\overline{\widetilde{\sigma}}, \Wig(g,f)))$,
therefore, viewed as operators independently of quantization rules,
all Cohen operators are Weyl operators (cfr. equation \eqref{Weylop}
and Proposition \ref{basic} (a)).

Due however to the freedom in the choice of the Cohen kernel
$\sigma$, we recapture by \eqref{QT} all types of quantizations used
in pseudo-differential calculus (Weyl, Kohn-Nirenberg,
localization, etc.). A particular family of operators of this kind
is considered in \cite{Bay11}, see Remark \ref{improve}.

When the symbol $a$ is the characteristic function of a measurable set in
$\Rdd$ it is natural to look at Cohen operators as a generalized way
of expressing the concentration of energy. In this spirit we shall
consider couples of these operators applied to a signal $f$ as the
substitute for the couple $(f,\widehat f)$ in the formulations of
the UP of Donoho-Stark type in Sections 3 and 5. More precisely in
Section 3 we shall consider the particular case of {\it localization
operators}, see \eqref{locop}, correcting a flaw in the estimate of
a Donoho-Stark type UP appearing in \cite{BogCarOli2016}, whereas in
Section 5 a similar UP in the general case of Cohen operators is
presented. Sections 2 and 4 are dedicated to some $L^p-$boundedness
results for Wigner (and Gabor) transforms and for general Cohen
class operators respectively, which are preliminary to the results
of the corresponding following sections.

Although a vast literature is available on $L^p-$boundedness, the norm estimates of Section 2
improve existing results as found in \cite{BoW} and \cite{Wong2004}, and those in Section 4 furnish extensions of results for Weyl operators to Cohen operators.

Concerning the meaning of ``concentration'' and ``not too
small'' sets we refer to the classical Donoho-Stark UP which we
recall next (see e.g. \cite{Gro01-1}, Thm. 2.3.1).
\begin{defn}\label{epscon}
Given $\varepsilon\geq 0$, a function $f\in L^{2}(\Rd)$ is
\emph{$\varepsilon$-concentrated} on a measurable set $U\subseteq
\Rd$ if
\[
\Big(\int_{\Rd\backslash
U}{|f(x)|^{2}dx}\Big)^{1/2}\leq\varepsilon\|f\|_{2}.
\]
\end{defn}

\begin{thm}[Donoho-Stark]\label{DS}
Suppose that $f\in L^{2}(\Rd)$, $f\neq 0$, is
$\varepsilon_{T}$-concentrated on $T\subseteq\Rd$, and $\widehat{f}$
is $\varepsilon_{\Omega}$-concentrated on $\Omega\subseteq\Rd$, with
$T,\Omega$ measurable sets in $\Rd$ and
$\varepsilon_T,\varepsilon_\Omega\geq 0$,
$\varepsilon_T+\varepsilon_\Omega \leq 1$. Then

\begin{equation}\label{DSineq}
|T||\Omega|\geq(1-\varepsilon_T - \varepsilon_{\Omega})^{2}.
\end{equation}
\end{thm}

As showed by Remark \ref{improve}, estimate \eqref{DSineq} is actually
improved by Theorem \ref{opL} of Section 2.

\section{$L^{p}$-continuity of the Gabor and Wigner distributions}

Sharp $L^p$-boundedness estimates for the Gabor transform (or \emph{short-time Fourier transform},
STFT)
$$
V_{g}f (x,\omega) = \int_{\Rd}{e^{-2\pi i \omega\cdot t}f(t)\overline{g(t-x)}dt}
$$
with window $g$, and applied to a signal $f$, shall be needed later on in this paper.
They are consequence of Young's inequality with optimal constants (Babenko-Beckner constants) which we recall here: if $f\in L^{p}(\Rd)$ and $g\in L^{q}(\Rd)$ and $\frac{1}{p} + \frac{1}{q} = 1 + \frac{1}{r}$, then $f\ast g\in L^{r}(\Rd)$ and \[\| f\ast g\|_{r}\leq(A_{p}A_{q}A_{r'})^{d}\|f\|_{p}\|g\|_{q},\] where $A_{p} = \left(\frac{p^{1/p}}{p'^{1/p'}}\right)^{1/2}$.

The boundedness of Gabor and Wigner transform has been widely studied, in several functional settings, cf. for example
\cite{BogDedOli2009}, \cite{CorNic2016}, \cite{Tof04-1}, \cite{Tof04-2}, \cite{Won98}. Here we focus on Legesgue spaces, for which an estimate with sharp constant can be found for the
one-dimensional case in \cite{Lieb}, whereas a more general result
but with no sharp estimate is proved in \cite{BogDedOli2009} (Proposition
3.1). We improve here the boundedness result of \cite{BogDedOli2009}
with an estimation of the constant of the type of \cite{Lieb}.

\noindent For $2\le p < +\infty$, $1\leq q < +\infty$, $p\geq \max\{q,q^\prime\}$, let us set
\begin{equation}\label{const}
H(p,q) = \left(\frac{q}{p}\right)^{d/p}
\frac{(p-q)^{(p-q)d/2pq}(qp-p-q)^{(qp-p-q)d/2pq}}{(q-1)^{(q-1)d/2q}(p-2)^{(p-2)d/2p}}.
\end{equation}
As it is easily verified, for every $q\in[1,+\infty)$, we have
$\lim_{p\to+\infty}H(p,q)=1$, therefore it is convenient to extend
\eqref{const} by setting $H(q,+\infty)=1$, and also
$H(+\infty,+\infty)=1$. With this agreement we have the following
result.

\begin{thm}\label{gengab}
Let $2\le p\le\infty$ and $p'\leq q \leq p$ (so that
also $p'\leq q'\leq p$). Then the Gabor transform defines
a bounded sesquilinear map
$$
V: (f,g)\in L^{q}(\Rd)\times L^{q'}(\Rd)\rightarrow V_gf \in
L^{p}(\R2d),$$
and
\begin{equation}\label{Vgcont}
\|V_{g}f\|_{p}\leq H(p,q)\|f\|_q \|g\|_{q'}.
\end{equation}
\begin{proof}
The case $p = +\infty$ is a trivial application of Young's
inequality which immediately yields the estimate $ \|V_g
f\|_{\infty}\leq \|f\|_{q}\|g\|_{q'}$. Suppose now that
$p<+\infty$. We recall that the Gabor transform can be written as
Fourier transform of a product (cf. \cite{Gro01-1}, Thm. 3.3.2),
namely $V_g f(x,\omega) =
(f\cdot\tau_{x}\overline{g})^{\widehat{}}(\omega)$, where the translation operator $\tau_x$ is defined as $\tau_x h(t):=h(t-x)$. We have the following estimation:
\begin{equation}
\begin{split}
\|V_g f\|_{p} &= \left(\int_{\Rd}{\left(\int_{\Rd}{|(f\cdot\tau_{x}\overline{g})^{\widehat{}}(\omega)|^p d\omega}\right)^{\frac{1}{p}\cdot p} dx}\right)^{1/p} \nonumber\\
&=\left(\int_{\Rd}{\|(f\cdot\tau_{x}\overline{g})^{\widehat{}}(\omega)\|_{p}^{p}dx}\right)^{1/p} \nonumber\\
&\leq \left(\int_{\Rd}{A_{p'}^{pd}\|f\cdot\tau_{x}\overline{g}\|_{p'}^{p}dx}\right)^{1/p}\nonumber\\
&= A_{p'}^{d}\left(\int_{\Rd}{\left(\int_{\Rd}{|f(y)|^{p'}|\overline{\widetilde{g}}(x-y)|^{p'}dy}\right)^{p/p'}dx}\right)^{1/p} \nonumber\\
&= A_{p'}^{d}\left(\int_{\Rd}{\left(|f|^{p'}\ast|\overline{\widetilde{g}}|^{p'}(x)\right)^{p/p'}dx}\right)^{\frac{1}{p}\frac{p'}{p'}}\nonumber\\
&= A_{p'}^{d}\||f|^{p'}\ast|\overline{\widetilde{g}}|^{p'}\|_{p/p'}^{1/p'}, \nonumber
\end{split}
\end{equation}
where $\widetilde{g}(x) = g(-x)$. Let us now apply the Young's inequality to $|f|^{p'}$ and $|\overline{\widetilde{g}}|^{p'}$ which are in $L^{\frac{q}{p'}}$ and $L^{\frac{q'}{p'}}$ respectively, and denote $s = \frac{q}{p'} = \frac{q(p-1)}{p}$, $t = \frac{q'}{p'} = \frac{q(p-1)}{p(q-1)}$, and $r = \frac{p}{p'} = p-1$. Then
\begin{equation}
\begin{split}
A_{p'}^{d}\||f|^{p'}\ast|\overline{\widetilde{g}}|^{p'}\|_{p/p'}^{1/p'} &\leq A_{p'}^{d}\left(A_{\frac{q(p-1)}{p(q-1)}}A_{\frac{q(p-1)}{p}}A_{\frac{p-1}{p-2}}\right)^{d/p'}\||f|^{p'}\|_{s}^{1/p'}\||\overline{\widetilde{g}}|^{p'}\|_{t}^{1/p'} \nonumber\\
&= H(p,q)\|f\|_{q}\|g\|_{q'},
\end{split}
\end{equation}
where we have
\begin{equation}
\begin{split}
H(p,q) &= A_{p'}^{d}\left(A_{\frac{q(p-1)}{p(q-1)}}A_{\frac{q(p-1)}{p}}A_{\frac{p-1}{p-2}}\right)^{d/p'} \nonumber\\
&=\left(\frac{q}{p}\right)^{d/p}\frac{(p-q)^{(p-q)d/2pq}(qp-p-q)^{(qp-p-q)d/2pq}}{(q-1)^{(q-1)d/2q}(p-2)^{(p-2)d/2p}}.\nonumber
\end{split}
\end{equation}
\end{proof}
\end{thm}
\begin{rem}
For $q=2$, we have $H(p,2) = \left(\frac{2}{p}\right)^{d/p}$,
which is the constant appearing in \cite{Gro01-1} (Sec.~3.3). In
the case $q=p$ the constant becomes $H(p,q) = H(p) =
\left(\frac{p^{(p-2)}}{(p-1)^{p-1}}\right)^{d/2p} =
\left(\frac{p'^{1/p'}}{p^{1/p}}\right)^{d/2}$, which is the
Babenko-Beckner constant $A_{p'}$.
\end{rem}

\noindent In view of the well-known formula
\begin{equation}\label{WigGab}
\Wig{}{}(f,g)(x,\omega)
= 2^d e^{4\pi i x\cdot\omega} V_{\widetilde{g}}f(2x,2\omega)
\end{equation}
(see for instance \cite{Gro01-1}, Lemma 4.3.1), we
have $\|\Wig{}{}(f,g)\|_{p} =
2^{\frac{p-2}{p}d}\|V_{\widetilde{g}}f\|_p$ and any $L^p-$
boundedness result for the Gabor transform automatically transfers to
a corresponding result for the Wigner transform. More explicitly:

\begin{prop}\label{wigp}
For $2\le p \le +\infty$ and $p'\leq q\leq p$, the Wigner
transform is a bounded map
$$
\Wig{}{}: L^{q}(\Rd)\times L^{q'}(\Rd)\rightarrow L^{p}(\R2d)
$$
and
\[
\|\Wig{}{}(f,g)\|_{p}\leq C(p,q)\|f\|_{q}\|g\|_{q'},
\]
where $C(p,q) = 2^{\frac{p-2}{p}d}H(p,q)$, and $H(p,q)$ is defined by \eqref{const}.
\end{prop}

\begin{rem}
In the particular case $q = p$, Proposition \ref{wigp} reads
$\Wig{}{}(f,g)\|_{p}\leq 2^{-d}4^{d/p}A_{p}^{d}\|f\|_{p}\|g\|_{p'},
$ a result which appears in Wong \cite{Wong2004}.

We also recall that the boundedness result of Proposition \ref{wigp}
without estimation of the boundedness constant appears in
\cite{BogDedOli2009}, Prop 3.1. From Prop. 3.2 of the same paper
\cite{BogDedOli2009} one can further deduce that in the remaining
cases for $q$ and $p$ we do not have boundedness. We can therefore
summarize the situation as follows.
\end{rem}

\begin{prop}\label{qq'p}
Let $p,q\in[1,\infty]$, then
$$
\Wig{}{}: L^{q}(\Rd)\times L^{q'}(\Rd)\rightarrow L^{p}(\R2d)
$$
is bounded if and only if $p'\le q \le p$. (Remark that this means
no boundedness for $p<2$).
\end{prop}

Motivated by Proposition \ref{qq'p}, it is natural to investigate
the cases
$$
\Wig{}{}: L^{r}(\Rd)\times L^{s}(\Rd)\rightarrow L^{p}(\R2d)
$$
where not necessarily $r,s$ are conjugate indices and, along these
lines, we present next a discussion that will yield to a complete
characterization of the cases of boundedness of the Wigner transform
on Lebesgue spaces, which, although not strictly needed in the
following sections, in our opinion has an interest in itself. The
same results, with suitably adapted constants, hold for the Gabor
transform.

We start by the case of non-boundedness on the diagonal of the
indices space.

\begin{prop}\label{noncont}
Let $p\in [1,\infty]$, then the map $\Wig{}{}: L^{q}(\Rd)\times
L^{q}(\Rd)\rightarrow L^{p}(\R2d)$ is not continuous if $q\neq 2$.
\begin{proof}
Consider a function $f\in \mathcal{S}(\Rd)$ and let $f_{\lambda}(x)=f(\lambda x)$ be the corresponding dilation by $\lambda>0$. Let $\varphi_{\lambda}(x) = \frac{f_{\lambda}(x)}{\|f_{\lambda}\|_{2}}$ be the normalization of $f_{\lambda}$, then an easy computation gives
\[
\|\varphi_{\lambda}\|_{q} = \frac{\lambda^{\left(\frac{1}{2}-\frac{1}{q}\right)d}}{\|f\|_2}\|f\|_{q}.
\]
Given $f$ as above,
\[
\Wig{}{}(\varphi_{\lambda})(x,\omega) = \frac{1}{\|f\|^{2}_{2}}\Wig{}{}(f)\left(\lambda x,\frac{\omega}{\lambda}\right),
\]
and $\Wig{}{}(\varphi_{\lambda})\in\mathcal{S}(\R2d)\subset L^{p}(\R2d)$. However, $\|\varphi_{\lambda}\|_{q}\rightarrow 0$ for $\lambda\rightarrow+\infty$ if $q<2$, whereas $\|\Wig{}{}(\varphi_{\lambda})\|_{p}$ is constant, as
\[
\|\Wig{}{}(\varphi_{\lambda})\|_{p} = \frac{\|\Wig{}{}(f)\|_{p}}{\|f\|^{2}_{2}}.
\]
Analogously, $\|\varphi_{\lambda}\|_{q}\rightarrow 0$ for $\lambda\rightarrow 0$ if $q>2$, but $\|\Wig{}{}(\varphi_{\lambda})\|_{p}$ still remain constant. Hence, we cannot have continuity of $\Wig{}{}:L^q\times L^q\to L^p$ for $q\neq 2$.
\end{proof}
\end{prop}

We extend now the non-boundedness result to the
general case.
\begin{prop}\label{noncont-interp}
Let $r,s,p\in [1,\infty]$, with $r\ne s'$, then the map $\Wig{}{}: L^{r}(\Rd)\times
L^{s}(\Rd)\rightarrow L^{p}(\R2d)$ is not continuous.
\end{prop}
\begin{proof}
Suppose, on the contrary, that there exists $p$ and $r\ne s'$ such that
$$
\Wig{}{}: L^{r}(\Rd)\times
L^{s}(\Rd)\rightarrow L^{p}(\R2d)
$$
is bounded. From the sesquilinearity of $\Wig$, we would then also have a bounded map
$$
\Wig{}{}: L^{s}(\Rd)\times
L^{r}(\Rd)\rightarrow L^{p}(\R2d).$$
Well-known interpolation theorems would therefore yield a bounded map
$$
\Wig{}{}: L^{(\frac{\theta}{r}+\frac{1-\theta}{s})^{-1}}(\Rd)\times
L^{(\frac{\theta}{s}+\frac{1-\theta}{r})^{-1}}(\Rd)\rightarrow L^{p}(\R2d)
$$
for any $\theta\in[0,1]$. For $\theta=1/2$ we would then obtain boundedness in the case
$$
\Wig{}{}: L^{q}(\Rd)\times
L^{q}(\Rd)\rightarrow L^{p}(\R2d)$$
with $\frac{1}{q}=\frac{1}{2}(\frac{1}{r}+\frac{1}{s})$, which contradicts Proposition \ref{noncont}.
\end{proof}

Proposition \ref{qq'p} together with Proposition
\ref{noncont-interp} finally yield the following proposition which
provides the complete picture of the situation, and is more
conveniently expressed in geometric terms.
\begin{prop} Let $r,s,p\in[1,\infty]$, then
$$
\Wig{}{}: L^{r}(\Rd)\times L^{s}(\Rd)\rightarrow L^{p}(\R2d)
$$
is bounded if and only if the point $(\frac{1}{r};\frac{1}{s})$ lies
on the segment $\left[(\frac{1}{p};\frac{1}{p'}),
(\frac{1}{p'};\frac{1}{p})\right]$, which should be considered empty for $p<2$, as shown by the following
picture.
\end{prop}

{\centering
\includegraphics[scale=0.35]{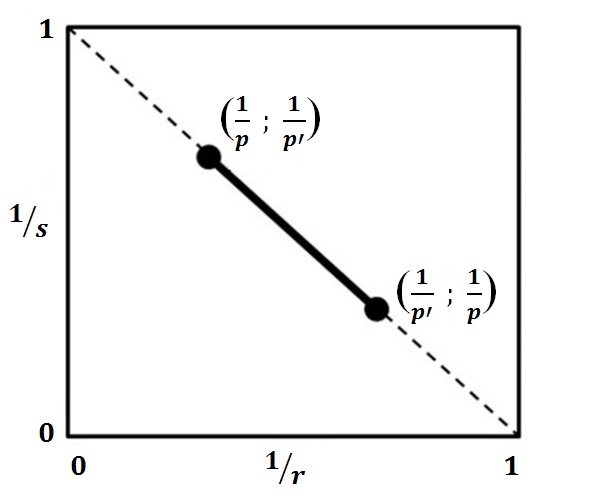}\\
}

\section{A Donoho-Stark UP for localization operators}

We revise in this section the content of Sections 2 and 3 of
\cite{BogCarOli2016} correcting a flaw in the proof of Theorem 6,
which expresses a Donoho-Stark uncertainty principle in terms of
localization operators and leads to an improvement of the classical
Donoho-Stark estimate, which now we specify correctly.

Let us first recall that a localization operator is a map of the
type:
\begin{equation}\label{locop}
f \longrightarrow L_{\phi,\psi}^{a}f =
\int_{\mathbb{R}^{{2d}}}a(x,\omega) V_{\phi}f(x,\omega)\,
\mu_\omega\tau_x\psi \, dx d\omega
\end{equation}
acting on $L^2(\Rd)$, with symbol $a\in L^q(\R2d)$, for
$q\in[1,\infty]$, and ``window'' functions $\phi,\psi\in L^2(\Rd)$.
Here $\mu_\omega\tau_x\psi(t)= e^{2\pi i \omega t}\psi(t-x)$ are {\it
time-frequency shifts} of $\psi(t)$. We refer to the available vast
literature (e.g. \cite{BoW, Coh95-1, FerGal10, Foll89,
Gro01-1, Won2002}) for the motivations and the meaning of these operators in
time-frequency and harmonic analysis, as well as for extensions to
more general functional settings. Their use in the Donoho-Stark UP
relies on the following boundedness estimate, which appears in
\cite{BogCarOli2016} (Lemma 4).
\begin{lem}\label{LemmaLocOp}
Let $\phi, \psi \in L^{2}(\Rd)$, $q\in [1,\infty]$ and consider the
quantization (see \eqref{locop}):
$$
L_{\phi,\psi}: a\in L^{q}(\mathbb{R}^{2d})\rightarrow
L_{\phi,\psi}^{a}\in B(L^2(\mathbb{R}^{d})).
$$
Then the following estimation holds
\[
\|L_{\phi,\psi}^{a}\|_{B(L^{2})}\leq
\left(\frac{1}{q'}\right)^{d/q'}\|\phi\|_{2}\|\psi\|_{2}\|a\|_{q},
\]
with $\frac{1}{q}+\frac{1}{q'}=1$, and setting
$(\frac{1}{q'})^\frac{1}{q'}=1$ for $q=1$.
\end{lem}

We shall use the previous result to obtain an uncertainty principle
involving localization operators in the special case where the
symbol is the characteristic function of a set, expressing therefore
{\sl concentration} of energy on this set when applied to signals in
$L^2(\Rd)$. In this case they are also known as {\sl concentration
operators}. The proof makes use of some tools from the
pseudo-differential theory which we now recall in the $L^2$
functional framework, for more general settings and references see
e.g. \cite{BogDedOli2010}, \cite{Hor90}, \cite{Won14}.

Given $f, g\in L^{2}(\Rd)$ we can associate an operator to the Wigner transform by using relation \eqref{QT}, and we call it \emph{Weyl pseudo-differential operators}:
\begin{equation}\label{Weylop}
(W^{a}f,g)_{L^{2}(\Rd)} = (a, \Wig(g,f))_{L^{2}(\Rdd)},
\end{equation}
where $a\in L^{2}(\Rdd)$. More explicitly this is a map of the type
$$
f\in L^2(\Rd)\longrightarrow W^a f(x)= \int_{\R2d} e^{2\pi i
(x-y)\omega} a\left(\frac{x+y}{2},\omega\right) f(y)\, dy\,
d\omega\in L^2(\Rd).
$$
\noindent The fundamental connection between Weyl and localization
operators is expressed by the formula which yields localization operators in terms of Weyl operators:
\begin{equation}\label{locWeyl}
L^a_{\phi,\psi}=W^b, \ \ \ \ \ \ \ {\rm with} \ \
b=a*\Wig(\widetilde\psi,\widetilde\phi),
\end{equation}
\noindent with $\psi,\phi \in L^2(\mathbb R^d)$ and where, for a generic function $u(x)$, we use the notation $\widetilde u(x)=u(-x)$.

\noindent Of particular importance for our purpose will be the fact
that Weyl operators with symbols $a(x,\omega)$ depending only on
$x$, or only on $\omega$, are multiplication operators, or Fourier
multipliers respectively. More precisely we have

\begin{equation}\label{multop-Fmultop}
\begin{array}{ll}
a(x,\omega)=a(x) \ \ \ \Longrightarrow \ \ \ W^af(x)=a(x)f(x) \\
a(x,\omega)=a(\omega) \ \ \ \Longrightarrow \ \ \ W^af(x)= \mathcal
F^{-1}[a(\omega) \widehat f(\omega)](x).
\end{array}
\end{equation}

\noindent We fix now some notations.

\noindent Let $T\subseteq \mathbb{R}_{x}^{d}$,
$\Omega\subseteq\mathbb{R}_{\omega}^{d}$ be measurable sets, and
write for shortness $\chi_T=\chi_{T\times\mathbb{R}^d}$ and
$\chi_\Omega=\chi_{\mathbb{R}^d\times\Omega}$, in such a way that
$\chi_T=\chi_T(x)$ and $\chi_\Omega=\chi_\Omega(\omega)$.

\noindent For $\lambda>0$, \ let \ $h_{\lambda}(x) =e^{-\pi\lambda x^{2}}$, \hskip0,3cm
$\Phi_{\lambda} = h_{\lambda}/ \| h_{\lambda} \|_2$, \hskip0,3cm $\varphi_{\lambda} =
h_{\lambda}/ \| h_{\lambda}\|_1$.

\noindent Moreover, for $\lambda_1,\lambda_2>0$, let
\begin{equation}\label{op1}
L_{1}f = L_{\Phi_{\lambda_1}}^{\chi_{T}}f =
\int_{\mathbb{R}^{2d}}{\chi_{T}(x)V_{\Phi_{\lambda_1}}f(x,\omega)\mu_{\omega}\tau_{x}\Phi_{\lambda_1}dxd\omega}
\end{equation}
and
\begin{equation}\label{op2}
L_{2}f = L_{\Phi_{\lambda_2}}^{\chi_{\Omega}}f =
\int_{\mathbb{R}^{2d}}{\chi_{\Omega}(\omega)V_{\Phi_{\lambda_2}}f(x,\omega)\mu_{\omega}\tau_{x}\Phi_{\lambda_2}dxd\omega}
\end{equation}
be the two localization operators with symbols $\chi_T, \chi_\Omega$
and windows $\Phi_{\lambda_1}, \Phi_{\lambda_2}$ respectively. We can state now the main
result of this section which is an UP involving the
$\varepsilon$-concentration of these two localization operators and
is the corrected version of \cite{BogCarOli2016}, Thm. 6.

\begin{thm}\label{opL}
With the previous assumptions on $T$, $\Omega$, $L_1$, $L_2$,
suppose that $\varepsilon_{T},\varepsilon_{\Omega}>0$,
$\varepsilon_T+\varepsilon_\Omega \leq 1$, and that $f\in
L^{2}(\mathbb{R}^{d})$ is such that
\begin{equation}\label{Hp}
\|L_{1}f\|_{2}^{2}\geq(1-\varepsilon_{T}^{2})\|f\|_{2}^{2}
\quad
\text{and}\quad\|L_{2}f\|_{2}^{2}\geq(1-\varepsilon_{\Omega}^{2})\|f\|_{2}^{2}.
\end{equation}
Then
\begin{equation}\label{ours}
|T||\Omega|\geq \sup_{r\in [1,\infty)}
(1-\varepsilon_{T}-\varepsilon_{\Omega})^{2r}(2r)^{-d}\left(\frac{(r+1)^{r+1}}{(r-1)^{r-1}}\right)^{d/2}.
\end{equation}
\end{thm}

\begin{proof}
Writing the operators $L_{j}$, $j=1,2$, defined in \eqref{op1} and
\eqref{op2} as Weyl operators we have:
\[
L_{1}f = W^{F_1}f, \quad \text{with}\quad F_1(x,\omega) =
\left(\chi_{T}(x)\otimes 1_{\omega}\right)\ast
\Wig{}{}(\Phi_{\lambda_1})(x,\omega)
\]
\[
L_{2}f = W^{F_2}f, \quad \text{with}\quad F_2(x,\omega) =
\left(1_x\otimes \chi_{\Omega}(\omega)\right)\ast
\Wig{}{}(\Phi_{\lambda_2})(x,\omega).
\]
An explicit calculation yields:
$$
\Wig(\Phi_{\lambda_j})(x,\omega) =
\varphi_{2\lambda_j}(x)\varphi_{2/\lambda_j}(\omega), \quad (j =
1,2).
$$
Therefore we have
$$
\ \ F_1(x,\omega)=(\chi_T\ast\varphi_{2\lambda_1})(x),
$$
$$
F_2(x,\omega)=(\chi_\Omega\ast\varphi_{\frac{2}{\lambda_2}})(\omega)
$$
in particular, $F_1$ depends only on $x$, and $F_2$ only on $\omega$.

It follows that
\[
L_{1}f = W^{F_1}f = F_1f,
\]
i.e. $L_1$ is the multiplication operator by the function $F_1$ and
\[
L_{2}f = W^{F_2}f = \mathcal{F}^{-1}F_2\mathcal{F}f,
\]
i.e. $L_2$ is the Fourier multiplier with symbol $F_2$. Now, for $j
= 1, 2$, we compute
\begin{equation}
\begin{split}
\|f\|_{2}^{2} &= \|(f - L_{j}f) + L_{j}f\|_{2}^{2}\\
&=((f-L_{j}f) + L_{j}f,(f - L_{j}f) + L_{j}f)\\
&=\|f- L_{j}f\|_{2}^{2} + \|L_{j}f\|_{2}^{2} + (f - L_{j}f,L_{j}f) +
(L_{j}f, f-L_{j}f). \quad \label{ripre}
\end{split}
\end{equation}
Next we show that $(f - L_{j}f,L_j f)\geq 0$. For $j = 1$ we have
\begin{equation}
\begin{split}
(f - L_{1}f, L_{1}f) &= (f,L_{1}f) - (L_{1}f,L_{1}f)\nonumber\\
&=\int{f\overline{F_1}\overline{f}}-\int{F_1f\overline{F_1}\overline{f}}\nonumber\\
&=\int{(1-F_1)\overline{F_1}|f|^{2}}\geq 0,
\end{split}
\end{equation}
as $F_1=\chi_T*\varphi_{2\lambda_1}$ is real, non negative, and
$\|F_1\|_{\infty} \le \| \chi_T\|_{\infty}
\|\varphi_{2\lambda_1}\|_1 = 1$.

Analogously, if $j = 2$ we have
\begin{equation}
\begin{split}
(f - L_{2}f, L_{2}f) &= (f,L_{2}f) - (L_{2}f,L_{2}f)\nonumber\\
&=\left(f,\mathcal{F}^{-1}F_2\mathcal{F}f\right)-\left(\mathcal{F}^{-1}F_2\mathcal{F}f,\mathcal{F}^{-1}F_2\mathcal{F}f\right)\nonumber\\
&=(\widehat{f},F_2\widehat{f}) - (F_2\widehat{f},F_2\widehat{f}) \nonumber\\
&=\int{\widehat{f}\overline{F_2\widehat{f}}} - \int{F_2\widehat{f}\overline{F_2\widehat{f}}}\nonumber\\
&=\int{(1-F_2)\overline{F_2}|\widehat{f}|^{2}}\geq 0,
\end{split}
\end{equation}
as $F_2=(\chi_\Omega\ast\varphi_{\frac{2}{\lambda_2}})$ is real, non
negative, and $\|F_2\|_{\infty} \le \| \chi_\Omega\|_{\infty}
\|\varphi_{2/\lambda_2}\|_1 = 1$.

\noindent Now, from \eqref{ripre}, since $(f-L_{j}f,
L_{j}f)\geq 0$, it follows
\[
\|f\|_2^{2} = \|f-L_{j}f\|_2^{2}+\|L_{j}f\|_2^{2} + 2(f-L_{j}f,
L_{j}f)
\]
and hence
\begin{equation}\label{estLj}
\|f-L_{j}f\|_2^{2} \leq
\|f\|_2^{2}-\|L_{j}f\|_2^{2}.
\end{equation}

From the hypothesis and \eqref{estLj} we obtain
\begin{equation*}
\left\{
  \begin{array}{ll}
    \|f-L_{1}f\|_2^{2}\leq \varepsilon_{T}^{2}\|f\|_2^{2}&,\\
    \|f-L_{2}f\|_2^{2}\leq
\varepsilon_{\Omega}^{2}\|f\|_2^{2}&.
  \end{array}
\right.
\end{equation*}
Considering the composition of $L_{1}$ and $L_{2}$ we have
\begin{equation}
\begin{split}
\|f-L_{2}L_{1}f\|_{2}&\leq\|f-L_{2}f\|_{2}+\|L_{2}f-L_{2}L_{1}f\|_{2}\nonumber\\
&\leq\varepsilon_{\Omega}\|f\|_{2} + \|L_{2}\|\|f-L_{1}f\|_{2}\nonumber\\
&\leq\varepsilon_{\Omega}\|f\|_{2}+ 1\cdot\varepsilon_{T}\|f\|_{2}\nonumber\\
&= (\varepsilon_{\Omega}+ \varepsilon_{T})\|f\|_{2},
\end{split}
\end{equation}
where the estimation of the operator norm
$\|L_2\|_{B(L^{2})}\leq\|\Phi_{\lambda_2}\|_{2}^{2}\|\chi_{\Omega}\|_{\infty}
= 1$ is obtained applying Lemma \ref{LemmaLocOp} with $q = \infty$,
or directly. Then

\begin{equation}\label{L2L1}
\begin{array}{ll}
\|L_{2}L_{1}f\|_2&\geq\|f\|_{2} - \|f-L_{2}L_{1}f\|_{2}\\
&\geq\|f\|_{2} - (\varepsilon_{\Omega}+ \varepsilon_{T})\|f\|_{2}\\
&=(1-\varepsilon_{T}-\varepsilon_{\Omega})\|f\|_{2},
\end{array}
\end{equation}

\noindent We look now for an upper estimate of $\|L_2L_1f\|$. For $r,k\in[1,+\infty)$ we have:
\begin{equation}
\begin{split}
\| L_2 L_{1}f\|_{2}&=  \|F_2\cdot\widehat{L_{1}f}\|_2\le \|F_2\|_{2r}\|\widehat{L_{1}f}\|_{2r'} \nonumber\\
& \le A_{(2r')'}^{d}\|F_{2}\|_{2r}\|L_{1}f\|_{(2r')'}
\nonumber\\
& \le A_{(2r')'}^{d}\left\|\left(\chi_{\Omega}\ast
\varphi_{2/\lambda_2}
\right)\right\|_{2r}\|\chi_T*\varphi_{2\lambda_1}\|_{(2r')'k}\|f\|_{(2r')'k'}
\nonumber\\
& \le A_{(2r')'}^{d} \|\chi_{\Omega}\|_{2r}\left\|\varphi_{2/\lambda_2}\right\|_{1}\|\chi_T\|_{(2r')'k}\left\|\varphi_{2\lambda_1}\right\|_1\|f\|_{(2r')'k'},\nonumber\\
\end{split}
\end{equation}
where $A_p$ is defined at the beginning of Section 2.

But $\|\varphi_{2/\lambda_2}\|_{1}=\|\varphi_{2\lambda_1}\|_1=1$
and, choosing $k=r+1$ so that $(2r^\prime)^\prime k^\prime=2$ and
$(2r^\prime)^\prime k=2r$, we have

$$\| L_2 L_{1}f\|_{2} \le A_{(2r')'}^d \|\chi_{\Omega}\|_{2r}
\|\chi_{T}\|_{2r} \|f\|_{2}.
$$

A direct computation of the Babenko constant yields
\begin{equation*}
A_{(2r')'} =
\left(\frac{(2r)^{1/r}(r-1)^{(r-1)/2r}}{(r+1)^{(r+1)/2r}}\right)^{1/2},
\end{equation*}
and therefore we obtain
\begin{equation}\label{new}
\begin{split}
\|L_2 L_{1}f\|_{2} &\leq \left(\frac{(2r)^{1/r}(r-1)^{(r-1)/2r}}{(r+1)^{(r+1)/2r}}\right)^{d/2}
\|f\|_{2}\|\chi_{\Omega}\|_{2r}\|\chi_{T}\|_{2r} \\
&= \left(\frac{(2r)^{1/r}(r-1)^{(r-1)/2r}}{(r+1)^{(r+1)/2r}}\right)^{d/2}\|f\|_{2}|\Omega|^{1/2r}|T|^{1/2r}.
\end{split}
\end{equation}
Finally from \eqref{L2L1} and \eqref{new} we obtain that, for every
$r\in [1,+\infty)$

\begin{equation}
\begin{split}
1-\varepsilon_{\Omega} - \varepsilon_{T}&\leq\frac{\|L_{1}L_{2}f\|_{2}}{\|f\|_{2}}\nonumber\\
&\leq\left(\frac{(2r)^{1/r}(r-1)^{(r-1)/2r}}{(r+1)^{(r+1)/2r}}\right)^{d/2}|\Omega|^{1/2r}|T|^{1/2r},
\end{split}
\end{equation}
which yields
\begin{equation*}
|T||\Omega|\geq
\sup_{r\in[1,+\infty)}(1-\varepsilon_{T}-\varepsilon_{\Omega})^{2r}(2r)^{-d}\left(\frac{(r+1)^{r+1}}{(r-1)^{r-1}}\right)^{d/2}.
\end{equation*}
\end{proof}

\begin{rem}\label{improve}

(1) The result involves the couple $(L_1f, L_2f)$ and the rectangle
$T\times\Omega$ analogously to the Donoho-Stark UP which involves
the couple $(f,\widehat f)$ and the same rectangle.

\noindent (2) Similarly to Lieb UP, but unlike Donoho-Stark UP, the
estimate is dependent on the dimension $d$ (and improves by
increasing $d$).

\noindent (3) The estimate $|T||\Omega|\ge \sup_{r\in [1,\infty)}
(1-\varepsilon_T-\varepsilon_\Omega)^{2r}(2r)^{-d}\left(\frac{(r+1)^{r+1}}{(r-1)^{r-1}}\right)^{d/2}$
is stronger then the classical Donoho-Stark estimate. Indeed, for
any choice of $\varepsilon_T$, $\varepsilon_\Omega$, the inequality $(1-\varepsilon_T-\varepsilon_\Omega)^{2r}\frac{1}{(2r)^{d}}\left(\frac{(r+1)^{r+1}}{(r-1)^{r-1}}\right)^{d/2}
> (1-\varepsilon_T-\varepsilon_\Omega)^{2}$ \ \ can be rewritten as
\noindent $ (1-\varepsilon_T-\varepsilon_\Omega)>
(2r)^{\frac{d}{2(r-1)}}\left(\frac{(r-1)}{(r+1)^{\frac{r+1}{r-1}}}\right)^{\frac{d}{4}},$
whose right-hand side vanishes as $r\to 1^+$. For example if
$\varepsilon_T+ \varepsilon_\Omega=0.1$, from \eqref{DSineq} we get
$|\Omega||T|\ge 0.81$ independently from the dimension $d$, but from \eqref{ours} we
have $|\Omega||T|\ge 0.9138$ for $r=1.34$, $d=1$, and $|\Omega||T|\ge 1.1358$ for $r=1.6$, $d=2$, etc.

\noindent  However from Theorem \ref{opL} we can not directly affirm
that we have an improvement of the Donoho-Stark estimate because our
hypotheses are different. The fact that we actually have an improvement is shown
in \cite{BogCarOli2016}, Prop. 10.

\noindent (4) Fixing $r=1$ we have
$(1-\varepsilon_1-\varepsilon_2)^{2r}\frac{1}{(2r)^{d}}\left(\frac{(r+1)^{r+1}}{(r-1)^{r-1}}\right)^{d/2}
= (1-\varepsilon_1-\varepsilon_2)^{2}$ i.e. the lower bound given by
Theorem \ref{opL} and that of Donoho-Stark UP coincide.

\noindent (5) The hypotheses of Theorem \ref{opL} depend on the
existence of two parameters $\lambda_1, \lambda_2>0$, which however
do not appear in the conclusion. This is due to the fact that the window
functions $\Phi_{\lambda_1},\Phi_{\lambda_2}$ are $L^2-$ normalized so that the norm of
the composition $L_1 L_2$ does not depend on these parameters.

\noindent (6) An open question: In the Donoho-Stark UP the case
$\varepsilon_T=\varepsilon_\Omega=0$ is equivalent to $\supp f\subseteq
T, \ \supp \widehat f\subset \Omega$ and yields $|T||\Omega|\ge 1$,
which is trivial because actually from the Benedicks UP we have
$|T||\Omega|=+\infty$.

\noindent The corresponding case
$\varepsilon_T=\varepsilon_\Omega=0$, for Theorem \ref{opL}, yields
$|T||\Omega|\ge
\sup_{r\in[1,\infty)}\frac{1}{(2r)^{d}}\left(\frac{(r+1)^{r+1}}{(r-1)^{r-1}}\right)^{d/2}
= (e/2)^d \approx (1.36)^d $. Is this a meaningful estimate or
also in this case actually $|T||\Omega|=+\infty$?
\end{rem}

\section{Cohen operators: quantizations and boundedness}
As mentioned in the Introduction, for any given time-frequency
representation $Q_\sigma =\sigma*\Wig$ we can consider, by formula \eqref{QT}, the operator
$T^a_\sigma$ depending on the symbol $a$. The class of operators that we obtain contains
classical pseudodifferential operators, localization, Weyl and $\tau$-Weyl operators, see \cite{BogDedOli2009}, as well as many other kind of operators of pseudodifferential type, such as the ones associated with the Born-Jordan representation, see \cite{CorGosNic2016}, and the pseudo-differential operators considered in \cite{Bay11}.

The aim of this section is to introduce some basic facts about the
correspondence $a\to T_\sigma^a$ and establish a Lebesgue functional
setting where these operators act continuously.

The following proposition summarizes the relations of the operators $T^a_\sigma$ with Weyl
operators, Schwartz kernels and adjoints. For simplicity we suppose signals, symbols and Cohen kernels in the Schwartz spaces, letting to the reader the standard extensions to more general settings.

\begin{prop}\label{basic}
Let $f,g \in \mathcal{S}(\Rd)$, $\sigma\in\mathcal{S}(\R2d)$, $a\in \mathcal{S}(\R2d)$. Then
\begin{equation*}
\begin{array}{ll}
a)&
T_{\sigma}^{a} = W^{a\ast\overline{\widetilde{\sigma}}},
\hbox{\ \ where $W^{a\ast\overline{\widetilde{\sigma}}}$ is the Weyl operator with symbol $a\ast\overline{\widetilde{\sigma}}$}\\
&\hbox{(and $\widetilde\sigma(z)=\sigma(-z)$, $z\in\Rdd$).}
\\
\\
b)&
(T_{\sigma}^{a})^{\ast} = T_{\overline{\sigma}}^{\overline{a}}, \ \ \hbox{i.e. the adjoint of a Cohen operator is the Cohen}\\
&\hbox{ operator corresponding to conjugated sesquilinear form and symbol.}
\\
\\
c)&
(T_{\sigma}^{a}f,g) = (k,f\otimes \overline g),
\hbox{\ \ i.e. $k=A\mathcal F_2^{-1}[a*\overline{\widetilde\sigma}]$ is the Schwartz kernel}\\
&\hbox{ of the operator $T_\sigma^a$,\ where $A:\phi(x,t)\mapsto \phi\left(\frac{x+t}{2},x-t\right)$, and $\mathcal F_2$ is the}\\
&\hbox{ Fourier transform with respect to the second half of variables.}
\\
\end{array}
\end{equation*}
\end{prop}
We omit the proof which is a straightforward computation.

The previous proposition leads to an interesting general view on the behavior of the different rules of association symbol-operator, improperly called ``quantizations''. By (a) all quantizations can be seen as ``deformation'' of the Weyl quantization where the symbol at first undergoes a filtering by the convolution with a fixed kernel $\sigma$. In absence of filtering, i.e. when $\sigma =\delta$, we have the ``pure'' Weyl quantization.

Many characteristics of the quantization can be then read from the Cohen kernel $\sigma$. For example it is natural to ask whether the correspondence $a\to T_\sigma^a$ is a quantization in the ``classical'' sense, i.e. it assigns self-adjoint operators to real symbols. We see from (b) that this happens if and only if the kernel $\sigma$ is itself real. The fact that the Weyl correspondence is a quantization in the classical sense corresponds to the fact that the Dirac $\delta$ is real (a distribution $u\in \mathcal S'(\Rdd)$ is ``real'' if its value on real-valued test functions is real).

Furthermore, we see from (c) that the correspondence between Schwartz kernels and operators is nothing else than formula \eqref{QT} where, as  sesquilinear form, is taken the skew-tensor product $f\otimes\overline g$, which is not in the Cohen class.

An explicit computation shows that the expression of the operator $T_{\sigma}^{a}$ acting on a function $f$ is
\begin{equation}\label{expl}
T_{\sigma}^{a} f(t) = \iint_{\R2d\times\R2d}{e^{2\pi i \xi(t-u)}a(x,\omega)\overline{\sigma\left(x-\frac{t+u}{2}, \omega-\xi\right)}f(u)\,dx\, d\omega\, du\, d\xi},
\end{equation}
in particular we remark that it is a classical pseudo-differential operator with amplitude
\[
a_{Q}(u,t,\xi) := \iint_{\R2d}{a(x,\omega)\overline{\sigma\left(x-\frac{t+u}{2}, \omega-\xi\right)}d\omega\, dx},
\]
(see \cite{Shu91}), however we shall not make use of this formula in this context.

With the following property we furnish a Lebesgue functional setting for the action of the Cohen operators.

\begin{thm}\label{contT}
Let $a\in L^{r}(\R2d)$ and $\sigma\in
L^{s}(\R2d)$. Moreover, let $q\in [1,2]$ be such that $\frac{1}{r} +
\frac{1}{s} = 1 + \frac{1}{q}$, and $p\in[q,q^\prime]$. Then $T:(a,\sigma)\in
L^{r}(\R2d)\times L^s(\R2d)\rightarrow T^a_\sigma \in B(L^{p}(\Rd))$
is a continuous map, and
\begin{equation}
\|T^a_\sigma\|_{B(L^{p})}\leq (A_{r}A_{s}A_{q'})^{d}C(q',p)\|a\|_{r}\|\sigma\|_{s},
\end{equation}
where $C(q',p) = 2^{\frac{q'-2}{q'}d}H(q',p)$.
\begin{proof}
Let $f\in L^{p}(\Rd)$, and $g\in L^{p'}(\Rd)$. By \eqref{QT} we obtain the following estimations, where we use H\"older's inequality and Proposition \ref{wigp} first, Young's inequality with the Babenko-Beckner's constants after:
\begin{equation}
\begin{split}
|(T^a_\sigma f,g)| &= |(a\ast\overline{\widetilde{\sigma}},\Wig(g,f))| \nonumber\\
&\leq \|a\ast\overline{\widetilde{\sigma}}\|_{q}\|\Wig(g,f)\|_{q'}\nonumber\\
&\leq \|a\ast\overline{\widetilde{\sigma}}\|_{q}C(q',p)\|f\|_p \|g\|_{p'} \nonumber\\
&\leq (A_{r}A_{s}A_{q'})^{d}C(q',p)\|a\|_r \|\sigma\|_s \|f\|_p \|g\|_{p'}.\nonumber
\end{split}
\end{equation}
Then the continuity of the operator $T^a_\sigma$ easily follows from standard arguments.
\end{proof}

\end{thm}
\begin{rem}\label{12}
If we take $f,g \in L^{2}(\Rd)$, we can reformulate the previous result. More precisely, let $q\in[1,2]$, $a\in L^{r}(\Rd)$ and $\sigma\in L^{s}(\Rd)$. Then $T:(a,\sigma)\in L^{r}(\R2d)\times L^s(\R2d)\rightarrow T^a_\sigma \in B(L^{2}(\Rd))$, with $\frac{1}{r} + \frac{1}{s} = 1 + \frac{1}{q}$, is a continuous map, and
\begin{equation}
\|T^a_\sigma\|_{B(L^2)}\leq (A_{r}A_{s}A_{q'})^{d}\left(\frac{2^{q'-1}}{q'}\right)^{d/q'}\|a\|_{r}\|\sigma\|_{s}.
\end{equation}
\end{rem}

We also remark that similar types of operators are considered in
\cite{Bay11} where boundedness and Schatten-Von Neuman results are
obtained. However the estimates in \cite{Bay11} are not given in
terms of the Cohen kernel and the type of operators considered there
can not cover the totality of the Cohen class operators. Actually
the sesquilinear forms associated with the operators in \cite{Bay11}
enjoy the following uncertainty principle: if their support has
finite measure then at least one of the two entries is null,
(\cite{Bay11}, Thm. 1.4.3) and this is not shared by all the
sesquilinear forms in the Cohen class, see Example 5 in
\cite{BogFerGal}.

\section{Donoho-Stark UP within the Cohen Class}
In this section we give a new formulation of the classical Donoho-Stark UP in the context of the Cohen class, by using the boundedness results founded previously.

Consider the operator $T_{\sigma}^{a}$ associated with a generic time-frequency representation in the Cohen class, cf. \eqref{QT}. We have the following result, that is the analogous for Cohen operators of Theorem \ref{opL}.
\begin{thm}\label{DSWigner}
Let $T,\Omega$ be measurable sets in $\Rd$. Let $\chi_{T}$ be the function $\chi_{T\times\Rd}(x,\omega)=\chi_{T}(x)$ and $\chi_{\Omega}$ be the function $\chi_{\Rd\times\Omega}(x,\omega)=\chi_{\Omega}(\omega)$, with corresponding operators $T_{\sigma}^{\chi_T}$ and $T_{\sigma}^{\chi_\Omega}$ associated with a generic representation $Q_{\sigma}$. We suppose that the kernel $\sigma$ is such that $F_1(t) = \int_{T\times\Rd}{\overline\sigma(x-t,\omega)dxd\omega}$ and $F_2(\omega)=\int_{\Rd\times\Omega} \overline\sigma (x,\eta-\omega)\,dx\,d\eta$ are (real) non negative function satisfying $\|F_j\|_{\infty}\leq 1$, $j=1,2$. Assume that $T_\sigma^{\chi_T}$ and $T_\sigma^{\chi_\Omega}$ satisfy the $\varepsilon$-concentration condition, i.e.,
\[
\|T_{\sigma}^{\chi_T}f\|_{2}^{2}\geq (1-\varepsilon_{T}^{2})\|f\|_{2}^{2} \qquad \text{and}\qquad \|T_{\sigma}^{\chi_\Omega}f\|_{2}^{2}\geq (1-\varepsilon_{\Omega}^{2})\|f\|_{2}^{2};
\]
moreover, suppose that
\begin{equation}\label{minT}
\min\{\|T_{\sigma}^{\chi_T}\|_{B(L^2)},\|T_{\sigma}^{\chi_\Omega}\|_{B(L^2)}\}\leq 1.
\end{equation}
Define $G_1(t)=\int_{\Rd} \overline\sigma (-t,\omega)\,d\omega$ and $G_2(\xi)=\int_{\Rd} \overline\sigma (x,-\xi)\,dx$. Then for every $s_1,s_2,p_1,p_2,r\in [1,\infty]$ satisfying $1/s_j+1/p_j=1+1/2r$, $j=1,2$, we have
\begin{equation*}
\vert T\vert^{1/s_1} \vert\Omega\vert^{1/s_2}\geq (1-\varepsilon_T-\varepsilon_\Omega)\frac{1}{\left( A_{s_1}A_{s_2}A_{p_1}A_{p_2}A_{\frac{2r}{r+1}}A_{\frac{2r}{2r-1}}^2\right)^d \Vert G_1\Vert_{p_1}\Vert G_2\Vert_{p_2}}.
\end{equation*}
\end{thm}

\begin{rem}
Taking $s_1=s_2:=s$ and $p_1=p_2:=p$, the conclusion of Theorem \ref{DSWigner} becomes
\begin{equation*}
\vert T\vert \vert\Omega\vert\geq \sup_{\afrac{r\in[1,\infty)}{1/s+1/p=1+1/2r}} (1-\varepsilon_T-\varepsilon_\Omega)^s \frac{1}{\left( A_s^2 A_p^2 A_{\frac{2r}{r+1}}A_{\frac{2r}{2r-1}}\right)^{sd} \Vert G_1\Vert_{p}^s\Vert G_2\Vert_{p}^s}.
\end{equation*}
\end{rem}

\begin{proof}[Proof of Theorem \ref{DSWigner}]
Observe at first that from Prop. \ref{basic} (a), and \eqref{multop-Fmultop}, for every $f\in L^2$ we have
\begin{equation}\label{chiTchiOmega}
T_\sigma^{\chi_T}f(t)=F_1(t)f(t),\quad \text{and}\quad T_\sigma^{\chi_\Omega}f(t)=\mathcal{F}^{-1}_{\omega\to t} \left[ F_2(\omega)\widehat{f}(\omega)\right](t).
\end{equation}
We can write
\begin{equation}
\begin{split}
\|f\|_{2}^{2} &= \|(f - T_{\sigma}^{\chi_T}f) + T_{\sigma}^{\chi_T}f\|_{2}^{2}\\
&=((f-T_{\sigma}^{\chi_T}f) + T_{\sigma}^{\chi_T}f,(f - T_{\sigma}^{\chi_T}f) + T_{\sigma}^{\chi_T}f)\\
&=\|f- T_{\sigma}^{\chi_T}f\|_{2}^{2} + \|T_{\sigma}^{\chi_T}f\|_{2}^{2} + (f - T_{\sigma}^{\chi_T}f,T_{\sigma}^{\chi_T}f) + (T_{\sigma}^{\chi_T}f, f-T_{\sigma}^{\chi_T}f). \quad \label{ripre1}
\end{split}
\end{equation}
We now show that $(f - T_{\sigma}^{\chi_T}f,T_{\sigma}^{\chi_T} f)\geq 0$:
\begin{equation}
\begin{split}
(f - T_{\sigma}^{\chi_T}f, T_{\sigma}^{\chi_T}f) &= (f,T_{\sigma}^{\chi_T}f) - (T_{\sigma}^{\chi_T}f,T_{\sigma}^{\chi_T}f)\nonumber\\
&=\int{f\overline{F_1}\overline{f}}-\int{F_1f\overline{F_1}\overline{f}}\nonumber\\
&=\int{(1-F_1)\overline{F_1}|f|^{2}}\geq 0,
\end{split}
\end{equation}
since $F_1$ is real, non negative, and $\|F_1\|_{\infty}\leq 1$ by hypothesis.  Then, it follows
\[
\|f\|_2^{2} = \|f-T_{\sigma}^{\chi_T}f\|_2^{2}+\|T_{\sigma}^{\chi_T}f\|_2^{2} + 2(f-T_{\sigma}^{\chi_T}f,
T_\sigma^{\chi_T} f),
\]
and hence \begin{equation}\label{estLjT} \|f-T_{\sigma}^{\chi_T}f\|_2^{2} \leq
\|f\|_2^{2}-\|T_{\sigma}^{\chi_T}f\|_2^{2}.
\end{equation}
From the hypothesis and \eqref{estLjT} we obtain
\begin{equation*}
\|f-T_{\sigma}^{\chi_T}f\|_2^{2}\leq \varepsilon_{T}^{2}\|f\|_2^{2}.
\end{equation*}
Reasoning in the same way for $T_\sigma^{\chi_\Omega}$, we get
\begin{equation*}
\|f-T_{\sigma}^{\chi_\Omega}f\|_2^{2}\leq
\varepsilon_{\Omega}^{2}\|f\|_2^{2}.
\end{equation*}
We assume for simplicity that $\Vert T_\sigma^{\chi_\Omega}\Vert_{B(L^2)}\leq 1$, cf. \eqref{minT} (in the other case the proof is similar). We then have
\begin{equation}
\begin{split}
\|f-T_{\sigma}^{\chi_\Omega}T_{\sigma}^{\chi_T}f\|_{2}&\leq\|f-T_{\sigma}^{\chi_\Omega}f\|_{2}+\|T_{\sigma}^{\chi_\Omega}f-T_{\sigma}^{\chi_\Omega}T_{\sigma}^{\chi_T}f\|_{2}\nonumber\\
&\leq\varepsilon_{\Omega}\|f\|_{2} + \|T_{\sigma}^{\chi_\Omega}\|_{B(L^2)}\|f-T_{\sigma}^{\chi_T}f\|_{2}\nonumber\\
&\leq\varepsilon_{\Omega}\|f\|_{2}+ 1\cdot\varepsilon_{T}\|f\|_{2}\nonumber\\
&= (\varepsilon_{\Omega}+ \varepsilon_{T})\|f\|_{2}.
\end{split}
\end{equation}
Then
\begin{equation}
\begin{split}
\|T_{\sigma}^{\chi_\Omega}T_{\sigma}^{\chi_T}f\|_2&\geq\|f\|_{2} - \|f-T_{\sigma}^{\chi_\Omega}T_{\sigma}^{\chi_T}f\|_{2}\nonumber\\
&\geq\|f\|_{2} - (\varepsilon_{\Omega}+ \varepsilon_{T})\|f\|_{2}\nonumber\\
&=(1-\varepsilon_{T}-\varepsilon_{\Omega})\|f\|_{2},
\end{split}
\end{equation}
and, from this, it follows that
\begin{equation}\label{opnorm}
1-\varepsilon_{\Omega} - \varepsilon_{T}\leq \|T_{\sigma}^{\chi_\Omega}T_{\sigma}^{\chi_T}\|_{B(L^2)}.
\end{equation}
From \eqref{chiTchiOmega}, and using H\"older inequality we have, for every $r,k\in [1,\infty]$
\begin{equation*}
\begin{split}
\|T_{\sigma}^{\chi_\Omega}T_{\sigma}^{\chi_T}f\|_2 &=\| F_2\widehat{F_1f}\|_2 \\
&\leq A_{(2r^\prime)^\prime}^d \| F_2\|_{2r} \| F_1 f\|_{(2r^\prime)^\prime} \\
&\leq A_{(2r^\prime)^\prime}^d \| F_2\|_{2r} \| F_1\|_{(2r^\prime)^\prime k} \| f\|_{(2r^\prime)^\prime k^\prime}.
\end{split}
\end{equation*}
We choose $k$ such that $(2r^\prime)^\prime k^\prime=2$, that implies $(2r^\prime)^\prime k=2r$. Observe now that $F_1=\chi_T*G_1$ and $F_2=\chi_\Omega*G_2$. Then, for every $s_1$, $s_2$, $p_1$, $p_2$ such that $1/s_j+1/p_j=1+1/2r$, $j=1,2$, by Young inequality we get
$$
\|T_{\sigma}^{\chi_\Omega}T_{\sigma}^{\chi_T}f\|_2 \leq \left( A_{(2r^\prime)^\prime} A_{s_1} A_{s_2} A_{p_1} A_{p_2} A_{(2r)^\prime}^2\right)^d \|\chi_T\|_{s_1} \|\chi_\Omega\|_{s_2} \| G_1\|_{p_1} \| G_2\|_{p_2} \| f\|_2,
$$
that, together with \eqref{opnorm}, implies the thesis.
\end{proof}

\end{document}